\documentclass[10pt]{article}
\usepackage{amsfonts, epsfig, amsmath, amssymb,amsthm, color,mathabx}
\usepackage{mathrsfs}

\usepackage[english]{babel}

\usepackage[round]{natbib}   

\newcommand{\E}{\mathbf{E}}
\renewcommand{\P}{\mathbf{P}}

\renewcommand {\epsilon}{\varepsilon}

\newtheorem{thm}{Theorem}[section]

\newtheorem{lem}[thm]{Lemma}
\theoremstyle{definition}
\newtheorem{dfn}[thm]{Definition}

\newtheorem{rem}[thm]{Remark}
\newtheorem{exa}[thm]{Example}

\DeclareMathSymbol{\ophi}{\mathalpha}{letters}{"1E}

\newcommand{\e}{\varepsilon}

\renewcommand{\phi}{\varphi}

\newcommand{\be}{\begin{equation}}
\newcommand{\ee}{\end{equation}}
\newcommand{\ben}{\begin{equation*}}
\newcommand{\een}{\end{equation*}}

\newcommand{\ba}{\begin{equation}\begin{aligned}}
\newcommand{\ea}{\end{aligned}\end{equation}}
\newcommand{\ban}{\begin{equation*}\begin{aligned}}
\newcommand{\ean}{\end{aligned}\end{equation*}}

\DeclareMathOperator{\sign}{sign}

\newcommand{\ex}{\mathrm{e}}
\newcommand{\di}{\mathrm{d}}

\newcommand{\rF}{\mathscr{F}}

\newcommand{\bF}{\mathbb{F}}

\newcommand{\bI}{\mathbb{I}}

\newcommand{\bR}{\mathbb{R}}

\newfont{\cyrfnt}{wncyr10}
\def\J3{\cyrfnt{\rm \u{\cyrfnt I}}}
\def\j3{\cyrfnt{\rm \u{\cyrfnt i}}}

\allowdisplaybreaks[4]

\numberwithin{equation}{section}

\begin{document}
\title{Stochastic selection problem for a Stratonovich SDE with power non-linearity}


 
\author{
Ilya Pavlyukevich\footnote{Institut f\"ur Stochastik, Friedrich--Schiller--Universit\"at Jena, Ernst--Abbe--Platz 2, 
07743 Jena, Germany; ilya.pavlyukevich@uni-jena.de}\quad 
and\quad Georgiy Shevchenko\footnote{Kyiv School of Economics, Mykoly Shpaka 3, Kyiv 03113, Ukraine; gshevchenko@kse.org.ua}
}

\maketitle

\begin{abstract}
In our paper [Bernoulli 26(2), 2020, 1381--1409], we found all strong Markov solutions that spend zero time at $0$
of the Stratonovich stochastic differential equation $\di X=|X|^{\alpha}\circ\di B$, $\alpha\in (0,1)$.
These solutions have the 
form $X_t^\theta=F(B^\theta_t)$, where $F(x)=\frac{1}{1-\alpha}|x|^{1/(1-\alpha)}\operatorname{sign} x$ and $B^\theta$ is the skew Brownian motion
with skewness parameter $\theta\in [-1,1]$
starting at $F^{-1}(X_0)$. In this paper we show how an addition of small external additive noise $\e W$ restores uniqueness. 
In the limit as $\e\to 0$, we recover heterogeneous diffusion corresponding to the physically symmetric case $\theta=0$.
\end{abstract}


\smallskip

 
\maketitle

\section{Introduction}

In \cite{CherstvyCM-13}, the authors considered the so-called \emph{heterogeneous} diffusion process defined as a solution of the 
Stratonovich stochastic differential equation
\begin{equation}
\label{e:mainn}
X_t=x+\int_0^t |X_s|^\alpha\circ \di B_s
\end{equation}
with $\alpha\in\bR$ and $B$ being a standard Brownian motion. It is always assumed that 
$|x|^\alpha=|x|^\alpha\cdot \bI(x\neq 0)$, so that $|0|^\alpha=0$ for any $\alpha\in \bR$. 
This equation can be seen as a Stratonovich version of the famous diffusion 
\ba
\label{e:girsanov}
X_t^G=x+\int_0^t |X_s^G|^\alpha\, \di B_s
\ea
firstly studied by \cite{Girsanov-62}. It is well-known that equation \eqref{e:girsanov} 
has a unique strong solution for $\alpha\geq 1/2$ and has infinitely many solutions for $\alpha\in (0,1/2)$. A complete analysis 
of equation \eqref{e:girsanov}  can be found in 
Chapter 5 from 
\cite{CheEng05}.

It is clear that in the Stratonovich setting, the presence of the irregular point $\{0\}$ also affects the 
existence and uniqueness of solutions of \eqref{e:mainn}.

The physicists' approach to \eqref{e:mainn} exploited in the papers by \cite{CherstvyCM-13} and 
\cite{sandev2022heterogeneous} is the solution by regularization. Indeed, let $\alpha\in (0,1)$ and let $\{\sigma_\e\}$ be a family 
of smooth functions with integrable derivatives such that $\sigma_\e(x)>0$ and 
$|\sigma_\e(x)-|x|^\alpha|\to 0$ uniformly on compacts as $\e\to 0$. 
Consider a \emph{regularized equation} $\di Z^\e_t=\sigma_\e(Z^\e_t)\circ \di B_t$ 
that can be equivalently written in the It\^o form as
\ba
\label{e:Xsigma}
\di Z^\e_t=\sigma_\e(Z^\e_t)\, \di B_t  + \frac12 \sigma_\e(Z^\e_t) \sigma'_\e(Z^\e_t)\,\di t,\quad Z^\e_0=x.
\ea
There exists a unique strong solution $Z^\e$ to \eqref{e:Xsigma} that can be found explicitly 
with the help of the well-known Lamperti method. Indeed, let 
\ba
F_\e(x):=\int_0^x \frac{\di y}{\sigma_\e(y)},\quad x\in\bR. 
\ea
The function $F_\e$ is monotonically increasing and $C^1$-smooth with locally integrable second derivative.
The It\^o formula with generalized derivatives from \cite[Chapter 2, Section 10]{krylov2008controlled}
is applicable and yields
\ba
F_\e(Z^\e_t)=F_\e(x) + B_t
\ea
and hence 
\ba
\label{e:solZ}
Z^\e_t=F_{\e}^{-1}(B_t + F_\e(x)). 
\ea
Now one can argue that upon passing to the limit as $\e\to 0$, the processes $Z^\e$ converge to the solution $X$ of the original SDE \eqref{e:mainn}.
Unfortunately, the rigorous justification of this intuitively transparent procedure turns out to be tricky. Let us illustrate this by an example.

\begin{exa}
We assume that $\alpha\in (-1,1)$ and for brevity we set the initial value $x=0$.
We consider the family $\{\sigma_\e\}_{\e\in(0,1]}$ of the form
\ba
\sigma_\e(x)=\sqrt{|x|^{2\alpha}+\e^2},\quad x\in\bR,
\ea
see Figure \ref{f:1} (left).
\begin{figure}
\begin{center}
\includegraphics{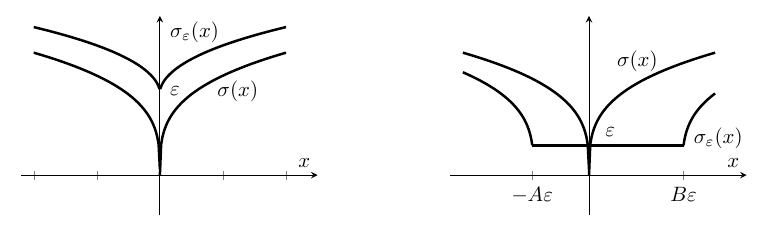}
\end{center}
\caption{The function $\sigma(x)=|x|^\alpha$ and $\sigma_\e(x)= \sqrt{|x|^{2\alpha}+\e^2}$ (left) and  $\sigma_\e(x)$ defined in \eqref{e:sigmae} (right).}
\label{f:1}
\end{figure}

Then, a straightforward calculation yields that
\ba
\label{e:FF}
F_\e(x)&\to F_0(x):=\frac{1}{1-\alpha}|x|^{1-\alpha}\sign x,\\
F^{-1}_\e(x)&\to F_0^{-1}(x):=|(1-\alpha)x|^{\frac{1}{1-\alpha}}\sign x ,
\ea
so that the limit process
\ba
X_t^0:= F_0^{-1}(B_t)
\ea
is a candidate for a solution to \eqref{e:mainn}

On the other hand, let $A,B\geq 0$ and let
\ba
\label{e:sigmae}
\sigma_\e(x)=\begin{cases}
              \e,\quad x\in [-A\e,B\e],\\
              |x-B\e+\e^{1/\alpha}|^\alpha,\quad x> B\e,\\
              |x+A\e-\e^{1/\alpha}|^\alpha,\quad x< -A\e,
             \end{cases}
\ea
see Figure \ref{f:1} (right).
Applying the Lamperti transformation and passing to the limit as above we find that
\ba
F^{-1}_\e(x)\to F^{-1}_{A,B}(x)=\begin{cases} 
0,\quad x\in[A,B],\\
    -|(1-\alpha)(x+A)|^{\frac{1}{1-\alpha}},\quad x< -A,\\
    |(1-\alpha)(x-B)|^{\frac{1}{1-\alpha}},\quad x>B.
                \end{cases}
\ea
Hence another candidate for the solution is the process 
\ba
X^{A,B}_t:=F^{-1}_{A,B}(B_t).
\ea
Note that such a process $X^{A,B}$ spends positive time at zero. 

It turns out that the processes $X^0$ and $X^{A,B}$ 
are indeed strong solutions of the equation \eqref{e:mainn} for $\alpha\in (-1,1)$. This follows by a straightforward verification with the 
help of the generalized It\^o formula by \cite{FPSh-95}.
\end{exa}

Excluding solutions spending positive time at zero as non-physical does not always guarantee uniqueness. Indeed, in our previous paper (\cite{PavShe-20}),
the following theorem was proven.

\begin{thm}[Theorem 4.5, \cite{PavShe-20}]
Let $x\in\bR$ and $\theta\in[-1,1]$, and
let $B^\theta$ be a skew Brownian motion that is a unique strong solution of the SDE
\ba
B^\theta_t=F^{-1}_0(x) + B_t +\theta L_t^0(B^\theta),\quad t\geq 0,
\ea
$L^0(B^\theta)$ being the symmetric local time of $B^\theta$ at zero.

\noindent
1. Let $\alpha\in (0,1)$. Then 
$X^\theta_t=F_0(B^\theta_t)$ 
is a
strong solution of \eqref{e:mainn} which is a homogeneous strong Markov process spending zero time at $0$.

The process $X^\theta$ is the unique
strong solution of \eqref{e:mainn} which is a homogeneous strong Markov process spending zero time at $0$ and such that
\ban
\P(X^\theta_t\geq 0\mid X_0 = 0) =\frac{1+\theta}{2},\quad t>0.
\ean
2. Let $\alpha\in (-1,0]$. Then $X^0_t=F(B_t^0)$ is the unique
strong solution of \eqref{e:mainn} which is a homogeneous strong Markov process spending zero time at $0$. 
\end{thm}

In other words, for $\alpha\in (0,1)$, the equation \eqref{e:mainn} is underdetermined even in the class of 
homogeneous strong Markov processes spending zero time at $0$. In this respect, equation \eqref{e:mainn} can be seen as a stochastic 
counterpart of the deterministic non-Lipschitzian equation $\dot x_t= |x_t|^\alpha$ that has infinitely many solutions if started at $x\leq 0$. 
 
In general, these solutions spend positive time at $0$. However one can always single out the unique \emph{maximal} solution spending zero time
at 0. Such a non-uniqueness feature of a ODE with continuous coefficients is known as the Peano phenomenon. It is also known that uniqueness, 
or singling out of physically meaningful solutions
spending zero time at $0$, 
can be restored by means of the \emph{regularization by noise}, studied for the first time by   
\cite{BaficoB-82} and \cite{veretennikov1983approximation}. 
In particular Theorem 5.1 in \cite{BaficoB-82} implies that solutions of the SDE
\ba
x_t^\e=x+\int_0^t |x^\e_s|^\alpha\,\di s + \e W_t
\ea
driven by a small Brownian motion $\e W$ weakly converge to the maximal deterministic solution
\ba 
x_t^*(x)= \begin{cases}
             \big((1-\alpha)t + x^{1-\alpha}\big)^\frac{1}{1-\alpha},\quad x>0,\\
             -\big(|x|^{1-\alpha} - (1-\alpha)t\big)_+^\frac{1}{1-\alpha}
             +\big((1-\alpha)t-|x|^{1-\alpha}\big)_+^{\frac{1}{1-\alpha}} ,\quad x\leq 0,
            \end{cases}
\ea
that spends zero time at $0$.
More results on the stochastic Peano phenomenon can be found in the works by 
\cite{delarue2014transition}, \cite{trevisan2013zero}, 
\cite{pilipenko2018selection}, \cite{PilipenkoPa-20}.

The goal of this paper is to apply the ``regularization'' procedure to the SDE \eqref{e:mainn}, i.e.\ to consider the perturbed SDE
\ba
\label{e:main+eps}
X_t^\e=x+\int_0^t |X_s^\e|^\alpha\circ \di B_s+\e W_t
\ea
in the presence of another independent Brownian motion $W$. From the physical point of view, the Brownian motion $W$ with vanishing amplitude $\e$
represents the \emph{ambient environmental noise} whereas the Brownian motion $B$ is responsible for the diffusive behaviour of the 
``heavy'' particle $X^\e$ in the non-homogeneous medium determined by the non-linear ``temperature profile'' $x\mapsto|x|^\alpha$. 

The qualitative result of the present paper is that solutions of the regularized equation \eqref{e:main+eps}
converge to the \emph{benchmark} solution
\ba
\label{e:benchmark}
X^0_t=F_0^{-1}(B_t)=\Big|(1-\alpha)B_t+ |x|^{1-\alpha}\sign x\Big|^\frac{1}{1-\alpha}\sign \Big((1-\alpha)B_t+ |x|^{1-\alpha}\sign x\Big)
\ea
that spends zero time at $0$ and has no ``asymmetry'' at the origin.   

This result is a manifestation of the ``selection'' procedure in the SDE case. In the limit as $\e\to 0$, the regularized solution
$X^\e$ always ``selects'' the natural solution \eqref{e:benchmark} among all possible solutions of \eqref{e:mainn} that may have asymmetric 
behaviour at 0, stay positive time at zero, or be non-Markovian, etc.

\section{Preliminaries and the main result \label{sec:prelim}}

Throughout the article, we work on a stochastic basis $(\Omega,\rF,\bF,\P)$, i.e.\ a complete 
probability space with a filtration $\bF = (\rF_t)_{t\ge 0}$ satisfying the standard assumptions. 
The process $(B,W) = (B_t, W_t)_{t\ge 0}$ is a standard continuous two-dimensional Brownian motion on this stochastic basis. 

First we briefly recall definitions related to stochastic integration. More details may be found in \cite{Protter-04}. 

The main mode of convergence considered here is the uniform convergence on compacts in probability 
(the u.c.p.\ convergence for short): a sequence $X^n = (X_t^n)_{t\ge 0}$, $n\geq 1$, of stochastic processes converges 
to $X = (X_t)_{t\ge 0}$ in u.c.p.\ if for any $t\ge 0$ 
\ba
\sup_{s\in[0,t]} |X^n_s - X_s| \stackrel{\P}{\to} 0,\quad  n\to\infty.
\ea

Let a sequence of (deterministic) partitions $D_n=\{0=t^n_0<t_1^n<t_2^n<\cdots\}=\{0=t_0<t_1<t_2<\cdots\}$ of $[0, \infty)$, $n\geq 1$, 
be such that for each $t\geq 0$ the number of points in each interval $[0, t]$ is finite, 
and the mesh $\|D_n\|:=\sup_{k\ge 1}|t_{k}^n - t_{k-1}^n|\to 0$ as $n\to \infty$.
A continuous stochastic process $X$ has quadratic variation $[X]$ along the sequence $\{D_n\}_{n\geq 1}$ if the limit
\ban
{}[X]_t:=\lim_{n\to \infty} \sum_{t_k\in D_n, t_k<t}|X_{t_{k+1}}-X_{t_{k}}|^2
\ean
exists in the u.c.p.\ sense.
Similarly, the quadratic 
covariation $[X,Y]$ of two continuous stochastic processes $X$ and $Y$ is defined as a limit in u.c.p.\
\ban
{}[X,Y]_t:=\lim_{n\to \infty} \sum_{t_k\in D_n, t_k<t}(X_{t_{k+1}}-X_{t_{k}})(Y_{t_{k+1}}-Y_{t_{k}}).
\ean
When $X$ and $Y$ are semimartingales, the quadratic variations $[X]$, $[Y]$ and the quadratic 
covariation $[X,Y]$ exist, moreover, they have bounded variation on any finite interval. 

Further, we define the It\^o integral as a limit 
in u.c.p.\
\ban
\int_0^t X_s\, \di Y_s=\lim_{n\to \infty} \sum_{t_k\in D_n, t_k<t} X_{t_{k}}(Y_{t_{k+1}}-Y_{t_{k}})
\ean
and the Stratonovich (symmetric) integral as a limit in u.c.p.\
\ban
&\int_0^t X_s\circ \di Y_s= \int_0^t X_s\, \di Y_s+\frac 12 [X,Y]_t  
\\&= \lim_{n\to \infty} \sum_{t_k\in D_n, t_k<t} \frac12 (X_{t_{k+1}}+X_{t_k})(Y_{t_{k+1}}-Y_{t_{k}}),
\ean
provided that both the It\^o integral and the quadratic variation exists. Again, when both $X$ and $Y$ are continuous semimartingales, 
both integrals exists, and the convergence holds in u.c.p.

We recall that a continuous semimartingale $X$ is an It\^o semimartingale if its
semimartingale characteristics are absolutely continuous with respect to Lebesgue measure, see, e.g., Definition 2.1.1 in \cite{jacod2011discretization} and 
\cite{ait2018semimartingale}.

In this paper, we define 
\ba
\sign x:=\begin{cases}
	-1,\ x< 0,\\
	0,\ x=0,\\
	1,\ x>0,
\end{cases}
\ea
and for any $\alpha\in \bR$ we set
\ba
|x|^\alpha:=\begin{cases}
	|x|^\alpha,\ x\neq 0,\\
	0,\ x=0.
\end{cases}
\ea
We also denote 
\ba
(x)^\alpha: = |x|^\alpha \sign x.
\ea
Throughout the article, $C$ will be used to denote a generic constant, whose concrete value is not important and may change between lines. 

First we give the intuition that will lead to the main result.
Although the function $x\mapsto|x|^\alpha$ is not smooth,  let us formally rewrite equation \eqref{e:main+eps} in the It\^o form
with the noise induced drift
in the spirit of the formula \eqref{e:Xsigma}. Thus we obtain another SDE
\ba
\label{e:main+eps-ito}
Y^\e_t = x + \int_0^t |Y^\e_s|^\alpha \di B_s + \frac{\alpha}{2}\int_0^t (Y^\e_s)^{2\alpha-1}\, \di s+ \e  W_t.
\ea
Note that the SDEs \eqref{e:main+eps} and \eqref{e:main+eps-ito} are 
of different nature and do not have to be equivalent.

Assume now that $Y^\e$ is the solution of 
\eqref{e:main+eps-ito}. Then it is a Markov process with the generator 
\ba
A^\e f(y)=\frac{|y|^{2\alpha}+\e^2}{2}f''(y)+ \frac{\alpha}{2} (y)^{2\alpha-1} f(y),\quad f\in C^2(\bR,\bR).
\ea
However, $A^\e$ is also the generator of the diffusion 
\ba
\label{e:Z}
Z^\e_t
& = x + \int_0^t \sqrt{|Z^\e_s|^{2\alpha}+\e^2}\, \di \widetilde B_s+ \frac{\alpha}{2}\int_0^t (Z^\e_s)^{2\alpha-1} \di s\\
& = x + \int_0^t \sqrt{|Z^\e_s|^{2\alpha}+\e^2}\circ \di \widetilde B_s
\ea
driven by some other Brownian motion $\widetilde B$. The solution to equation \eqref{e:Z} has been determined explicitly in \eqref{e:solZ}
as $Z^\e_t=F^{-1}_\e(\widetilde B_t+F_\e(x))$, and hence we get that 
\ba
Y^\e_t\stackrel{\di}{=} F^{-1}_\e(\widetilde B_t+F_\e(x)),\quad t\geq 0.
\ea
It is clear that $Z^\e$, and also $Y^\e$, converges in law to the benchmark solution $X^0$. The same convergence 
will also hold for solutions $X^\e$ of the perturbed Stratonovich equation \eqref{e:mainn} if one establishes an equivalence between the equations
 \eqref{e:mainn} and \eqref{e:main+eps-ito}.

Now let us turn to equations \eqref{e:mainn} and \eqref{e:main+eps}. The concept of strong solution for these equations is defined in a standard manner.
\begin{dfn}
A strong solution to \eqref{e:main+eps} is a continuous stochastic process $X^\e$ such that 
\begin{enumerate}
		\item
		$X^\e$ is adapted to the augmented 
		natural filtration of $(B,W)$;
		\item 
		for any $t\geq 0$, the integral $\int_0^t |X_s^\e|^\alpha\, \di B_s$ and the quadratic covariation 
		$[|X^\e|^\alpha,\ B]_t$ exist;
		\item 
		for any $t\geq 0$, equation \eqref{e:main+eps}  holds $\P$-a.s.
	\end{enumerate}
\end{dfn}

\begin{dfn}
A strong solution to \eqref{e:main+eps-ito} is a continuous stochastic process $Y^\e$ such that 
\begin{enumerate}
		\item
		$Y^\e$ is adapted to the augmented 
		natural filtration of $(B,W)$;
		\item 
		for any $t\geq 0$, the integrals $\int_0^t |Y^\e_s|^\alpha\, \di B_s$ and $\int_0^t (Y^\e_s)^{2\alpha-1}\, \di s$ exist;
\item 
for any $t\geq 0$, equation \eqref{e:main+eps-ito}  holds $\P$-a.s.
\end{enumerate}
\end{dfn}

\begin{dfn}
A weak solution to \eqref{e:main+eps} is a triple $(\widetilde X^\e,\widetilde B,\widetilde W)$
of adapted continuous processes on a
stochastic basis $(\widetilde \Omega,\widetilde \rF,\widetilde \bF,\widetilde \P)$ such that
\begin{enumerate}
\item
$\widetilde B$ and $\widetilde W$
are independent standard Brownian motions on $(\widetilde \Omega,\widetilde \rF,\widetilde \bF,\widetilde \P)$;
\item
 for any $t \geq 0$, the integral $\int_0^t |\widetilde X^\e_s|^\alpha\,\di \widetilde B_s$ 
 and the quadratic covariation $[|\widetilde X^\e|^\alpha ,\widetilde B]_t$ exist;
\item
 for any $t\geq 0$,
\ba
\widetilde X_t^\e = x+ \int_0^t |\widetilde X_s^\e|^\alpha\circ \di \widetilde B_s + \e \widetilde W_t
\ea 
holds $\widetilde \P$-a.s.
\end{enumerate}
\end{dfn}

\begin{dfn}
A weak solution to \eqref{e:main+eps-ito} is a 
 is a triple $(\widetilde Y^\e,\widetilde B,\widetilde W)$
of adapted continuous processes on a
stochastic basis $(\widetilde \Omega,\widetilde \rF,\widetilde \bF,\widetilde \P)$ such that
\begin{enumerate}
\item
$\widetilde B$ and $\widetilde W$
are independent standard Brownian motions on $(\widetilde \Omega,\widetilde \rF,\widetilde \bF,\widetilde \P)$;
\item
 for any $t \geq 0$, the integrals $\int_0^t |\widetilde Y^\e_s|^\alpha\,\di \widetilde B_s$ 
 and $\int_0^t (\widetilde Y^\e)^{2\alpha-1}\,\di s$ exist;
\item
 for any $t\geq 0$,
\ba
\widetilde Y_t^\e = x+ \int_0^t |\widetilde Y_s^\e|^\alpha\,\di \widetilde B_s 
+ \frac{\alpha}{2}\int_0^t (\widetilde Y^\e)^{2\alpha-1}\,\di s +\e \widetilde W_t
\ea 
holds $\widetilde \P$-a.s.
\end{enumerate}
\end{dfn}

The main results of this paper are given in the following theorems. The proofs will be presented in the subsequent sections.
\begin{thm}
\label{t:weakexists}
For any $\e>0$ and any $\alpha \in (-1,1)$, equation \eqref{e:main+eps-ito} has a weak solution, where 
for $\alpha = 0$, $\frac{\alpha}{2}\int_0^t (Y^\e_s)^{2\alpha-1} \di s :=0$ and
for $\alpha \in (-1,0)$ 
the integral 
\ba
\int_0^t (Y^\e_s)^{2\alpha-1}\,\di s:= \lim_{\delta\downarrow 0}\int_0^t (Y^\e_s)^{2\alpha-1}\bI(|Y^\e_s|>\delta)\, \di s
\ea
is understood in the principle value (v.p.) sense as defined in \cite{Cherny-01}.
\end{thm}

\begin{thm}
\label{t:unicite}
For any $\e>0$ and any $\alpha \in (0,1)$, the pathwise uniqueness property holds for the SDE \eqref{e:main+eps-ito}.
\end{thm}

Applying  the Yamada--Watanabe theorem, we arrive at the following result.

\begin{thm}
\label{t:exuniqito}
	For any $\e>0$ and $\alpha \in (0,1)$, equation \eqref{e:main+eps-ito} has a unique strong solution. 
\end{thm}

Concerning the equivalence of  \eqref{e:main+eps} and  \eqref{e:main+eps-ito} the following holds true.

\begin{thm}
\label{t:itobracket}
For any $\alpha\in (0,1)$, a strong solution to the stochastic differential equation 
\eqref{e:main+eps-ito} is also a strong solution to \eqref{e:main+eps}.
\end{thm}

To establish strong uniqueness of equation \eqref{e:main+eps}, we show that any of its strong solutions also solves \eqref{e:main+eps-ito} and appeal to Theorem~\eqref{t:exuniqito}. We will establish the equivalence of solutions that are It\^o semimartingales. In other words, we assume that the
bracket process $[|X^\e|^\alpha,B]$ is absolutely continuous with respect to Lebesgue measure.

\begin{rem}
 It will follow from Lemma~\ref{l:skobka} that for any strong solution $Y^\e$ to \eqref{e:main+eps-ito}, $\alpha\in (0,1)$, 
 the process $[|Y^\e|^\alpha,B]$ is absolutely continuous and hence $Y^\e$ is an It\^o semimartingale.
\end{rem}

\begin{thm}
\label{t:final}
Let $\e>0$ and let $X^\e$ be a strong solution to \eqref{e:main+eps} such that the bracket $[|X^\e|^\alpha,B]$ is absolutely continuous
with respect to Lebesgue measure. 
Then it is also a strong solution to \eqref{e:main+eps-ito}. In particular, strong existence and uniqueness holds for equation \eqref{e:main+eps} in the class of It\^o semimartingale solutions.
\end{thm}

Eventually we formulate the main qualitative result about the stochastic selection of the benchmark solution.  
\begin{thm}
\label{t:selection}
Let $\alpha\in (0,1)$.
Let $X^\e$ be a strong solution of \eqref{e:main+eps-ito} or an It\^o semimartingale solution of \eqref{e:main+eps}. Let $X^0$ be a benchmark solution 
\eqref{e:benchmark}.
Then 
\ba
X^\e\to X^0
\ea
in u.c.p.\ as $\e\to 0$.
\end{thm}

\section{Analysis of the SDE \eqref{e:main+eps-ito} with singular drift}

We first address the question of weak existence of a solution of equation \eqref{e:main+eps-ito}. It can be established even for $\alpha \in (-1,1)$. 
However, for negative $\alpha$, the drift should be understood in the principal value (v.p.), see \cite{Cherny-01}. 

The following function that has already appeared above in the Lamperti method 
will be crucial in showing both weak existence and pathwise uniqueness of solution.
Let
\ba
\label{e:sigma}
\sigma_\e(y)=\sqrt{|y|^{2\alpha}+\e^2}
\ea
and
\begin{equation}
\label{e:Fe}
F_\e(x) = \int_0^x \frac{\di y}{\sigma_\e(y)}, \quad x\in\bR,\ \e>0.
\end{equation}
Clearly, for each $\e>0$ and $\alpha\in(-1,1) $, the transformation $F_\e\colon \bR\to \bR$ is a bijection. In the neighbourhood of zero we have
\ba
\label{e:Fasympt}
F_\e(x)\approx \begin{cases}
                \e^{-1}\cdot x,\quad \alpha\in (0,1),\\
                (1+\e^2)^{-1/2}\cdot x,\quad \alpha=0,\\
                (1+|\alpha|)^{-1}\cdot  (x)^{1+|\alpha|} ,\quad \alpha\in(-1,0),
               \end{cases}
\ea
and
\ba
\frac{\di}{\di x} F_\e(x)=\sigma_\e(x)\approx \begin{cases}
                \e^{-1},\quad \alpha\in (0,1),\\
                (1+\e^2)^{-1/2}\cdot x,\quad \alpha=0,\\
                |x|^{|\alpha|} ,\quad \alpha\in(-1,0).
               \end{cases}
\ea
Hence, $F^{-1}_\e$ is a monotonically increasing absolutely continuous function, 
$\frac{\di}{\di x} F^{-1}_\e(x)=\sigma_\e(F_\e^{-1}(x))$, and in the neighbourhood of zero 
\ba
\frac{\di}{\di x} F^{-1}_\e(x)
\approx \begin{cases}
                \e ,\quad \alpha\in(0,1),\\
                (1+\e^2)^{1/2}\cdot x,\quad \alpha=0,\\
                (1+|\alpha|)^\frac{2|\alpha|}{1+|\alpha|}\cdot  |x|^\frac{2|\alpha|}{1+|\alpha|} ,\quad \alpha\in (-1,0).
               \end{cases}
\ea
In particular for 
$\alpha\in (-1,1)$, $\frac{\di}{\di x} F^{-1}_\e$ is locally square integrable, so that the generalized It\^o formula
\cite[Theorem 4.1]{FPSh-95} for the Brownian motion is applicable to
$F^{-1}_\e$.
 
For $\alpha\in (0,1)$, the second derivative
\ba
\frac{\di^2}{\di x^2} F_\e(x)=
-\frac{\alpha (x)^{2\alpha-1}}{(|x|^{2\alpha}+\e^2)^{3/2}},\quad x\neq 0,
\ea
is in $L^1_{\text{loc}}(\bR)$, so that the It\^o formula with generalized derivatives from \cite[Chapter 2, Section 10]{krylov2008controlled}
is applicable.

\subsection{Proof of Theorem \ref{t:weakexists}\label{s:wexist}}

In this section we assume that $\alpha\in(-1,1)$.
Let $\e>0$ be fixed. 

Let $W^1$, $W^2$ be independent standard Wiener processes, and let 
$\widehat W: = \frac1{\sqrt2} (W^1 + W^2)$. Define the process
\begin{equation}
\label{e:wt X}
\widetilde Y_t^\e = F_\e^{-1} \Big( F_\e(x) + \widehat W_t\Big)
\end{equation}
	and set 
\begin{align}
\label{e:wt B}
\widetilde B_t^\e &= \frac{1}{\sqrt{2}} 
\int_0^t \frac{\bigl(|\widetilde{Y}^\e_s|^\alpha + \e\bigr)\,\di W^1_s 
+ \bigl(|\widetilde{Y}^\e_s|^\alpha - \e\bigr)\,\di W^2_s}{\sigma_\e(\widetilde Y_s^\e)},\\
\label{e:wt W}
\widetilde W_t^\e & = \frac{1}{\sqrt{2}} \int_0^t \frac{\bigr(\e-|\widetilde{Y}^\e_s|^\alpha \bigr)\,\di W^1_s 
+ \bigl(|\widetilde{Y}^\e_s|^\alpha + \e\bigr)\,\di W^2_s}{\sigma_\e(\widetilde Y_s^\e)}.
\end{align}
It is straightforward to check that $\widetilde B^\e$ and $\widetilde W^\e$ are continuous martingales with 
$[\widetilde B^\e]_t = [\widetilde W^\e]_t = t$, $[\widetilde B^\e,\widetilde W^\e]_t = 0$. Then, by the L\'evy characterization, $\widetilde B^\e$ and 
$\widetilde W^\e$ are independent standard Wiener processes. It is also clear that
\ba
\label{dW = |X|dB+edW}
|\widetilde Y_t^\e|^\alpha\, \di \widetilde B_t^\e + \e \di \widetilde W_t^\e
&= \frac{|\widetilde Y_t^\e|^\alpha}{\sqrt{2}} \frac{(|\widetilde{Y}^\e_t|^\alpha + \e )\,\di W^1_t + (|\widetilde{Y}^\e_t|^\alpha - \e)\,\di W^2_t}
{\sigma_\e(\widetilde Y_t^\e)}\\
&+ \frac{\e}{\sqrt{2}} \frac{(\e-|\widetilde{Y}_t^\e|^\alpha)\,\di W^1_t + (|\widetilde{Y}_t^\e|^\alpha + \e)\,\di W^2_t}{\sigma_\e(\widetilde Y_t^\e)}\\
&=\sigma_\e(\widetilde Y_t^\e)\,\di \widehat W_t .
\ea
Applying the generalized It\^o formula (Theorem 4.1 in \cite{FPSh-95}) to the Brownian motion 
$\widehat W$ and the function $F^{-1}_\e$ in \eqref{e:wt X} and taking into account
\eqref{dW = |X|dB+edW}, we get
\ba
\widetilde Y_t^\e & = x + \int_{0}^{t} \sigma_\e(\widetilde Y_s^\e) \circ\di \widehat W_s \\
& = x + \int_{0}^{t}\sigma_\e(\widetilde Y_s^\e)\, \di \widehat W_s 
+ \frac12 [\sigma_\e(\widetilde Y^\e),\widehat W  ]_t \\
& = x + \int_{0}^{t} |\widetilde Y_s^\e|^{\alpha}\, \di \widetilde B_s^\e + \e \widetilde W_s^\e 
+  \frac12 [\sigma_\e(\widetilde Y^\e),\widehat W  ]_t.
\ea
For  $\alpha\in (0,1)$, by \cite[Remark 3.2 (c)]{FPSh-95},   
\ba
{}[\sigma_\e(\widetilde Y^\e),\widehat W ]_t = \alpha \int_0^t \bigl(\widetilde Y^\e_s\bigr)^{2\alpha-1}\, \di s;
\ea
for $\alpha \in (-1,0)$, by \cite[Corollary 4.4]{Cherny-01},
\ba
{}[\sigma_\e(\widetilde Y^\e),\widehat W ]_t = \alpha\cdot \mathrm{v.p.}\int_0^t (\widetilde Y^\e_s)^{2\alpha-1}\, \di s,
\ea
whence we arrive at the SDE \eqref{e:main+eps-ito} for the triple $(\widetilde{Y}^\e$, $\widetilde B^\e$, $\widetilde W^\e)$.

\subsection{Proof of Theorem \ref{t:unicite}}

Now we turn to the path-wise uniqueness of solution to \eqref{e:main+eps-ito}, using the ideas by \cite{legall1984one}.
In this section we assume that $\alpha\in(0,1)$.
 
Let $Y^\e$, $\widetilde Y^\e$ be two strong solutions to \eqref{e:main+eps-ito} and let 
$U^\e: = F_\e(Y^\e)$ and $\widetilde U^\e: = F_\e(\widetilde Y^\e)$. 
Then, by the extension of It\^o's formula from \cite[Chapter 2, Section 10]{krylov2008controlled} we have
\ba
\label{e:U}
\di U_t^\e& =\di F_\e(Y^\e_t)  = G_\e(Y_t^\e)\,\di B_t + H_\e(Y^\e_t)\, \di W_t\\
&= G_\e\bigl(F_\e^{-1}(U^\e_t)\bigr)\, \di B_t + H_\e\bigl(F_\e^{-1}(U^\e_t)\bigr)\,\di W_t, \quad t\ge 0,
\ea
where 
\ba
\label{e:ggg}
G_\e(y) = \frac{|y|^\alpha}{\sigma_\e(y)},\quad H_\e(y) = \frac{\e}{\sigma_\e(y)}.
\ea
The same formula holds for the processes $\widetilde Y^\e$ and $\widetilde U^\e$, too.

Since $G_\e(y)^2 + H_\e(y)^2\equiv 1$, both $U^\e$ and $\widetilde U^\e$ are standard Wiener processes, by virtue of the L\'evy characterization.
This, in particular, implies weak uniqueness for \eqref{e:main+eps-ito}.

Denote $R_t^\e = U_t^\e -\widetilde U_t^\e$. Then by Tanaka's formula
\ba
U^\e_t\vee \widetilde U^\e_t & = U^\e_t + (\widetilde U^\e_t-U^\e_t)^+\\
&= U^\e_t+ \int_0^t \bI(\widetilde U^\e_s>U^\e_s)\, \di (\widetilde U^\e_s-U^\e_s)+  \frac12 L^0_t(\widetilde U^\e-U^\e )\\
&=x +\int_0^t G_\e\bigl(F_\e^{-1}(U^\e_s\vee \widetilde U^\e_s)\bigr)\, \di B_s + H_\e\bigl(F_\e^{-1}(U^\e_s\vee \widetilde U^\e_s)\bigr)\,\di W_s  +\frac12 L^0_t(R^\e).
\ea
Suppose that $L^0_t(R^\e) = 0$. 
Then, by appealing to the extended It\^o formula once more, we see that
$F_\e(U^\e\vee \widetilde U^\e)$ is another strong solution to \eqref{e:main+eps-ito}, which implies that $U^\e = \widetilde U^\e$ in view of weak uniqueness. 

Thus, we are left to show that $L^0_t(R^\e) = 0$. 
To this end, it suffices to show that for any $t\geq 0$
\begin{equation}
\label{e:int 1/z<oo}
\int_{0}^t \frac{\bI(R^\e_s>0)}{R^\e_s}\,  \di[R^\e]_s<\infty\quad \text{a.s.,}
\end{equation}
see Lemma 1.1 in \cite{legall1984one}. Indeed, since $R^\e$ is a continuous martingale 
the mapping $a\mapsto L^a_t(R^\e)$ is continuous, see Corollaries 1.8 and 1.9 in Chapter IV in \cite{RevuzYor05}.
Therefore if $L^0_t(R^\e)=L^{0+}(R^\e)$ does not vanish 
the occupation time formula will yield
\ba
\int_0^t \frac{\bI(R^\e_s>0)}{R^\e_s}\,\di[R^\e]_s= \int_{0+}^\infty \frac{L^a_t(R^\e)}{a}\,\di a = +\infty \quad \text{a.s.}
\ea
To show \eqref{e:int 1/z<oo} we apply forthcoming Lemma \ref{l:G} to get
\ba
\label{e:gg}
\int_0^t \frac{\bI(R^\e_s>0) \, \di[R^\e]_s}{R^\e_s} & 
= \int_0^t \frac{\bigl|G_\e(Y^\e_s) - G_\e(\widetilde Y_s^\e)\bigr|^2+\bigl|H_\e(Y_s^\e) - H_\e(\widetilde Y_s^\e)\bigr|^2} 
{U^\e_s - \widetilde U^\e_s} \bI(R^\e_s>0)\,\di s\\
&\leq C_\e \int_0^t \frac{\bigl| |Y_s^\e|^\alpha - |\widetilde Y_s^\e|^\alpha\bigr|^2\wedge 1} {U^\e_s - \widetilde U^\e_s}  \bI(R^\e_s>0)\,\di s \\
&\leq C_\e \int_0^t \frac{\bigl|(Y_s^\e)^\alpha - (\widetilde Y_s^\e)^\alpha\bigr|^2 \wedge 1} {U^\e_s - \widetilde U^\e_s}  \bI(R^\e_s>0)\,\di s\\
&\leq C_\e \int_0^t \frac{(Y_s^\e)^\alpha - (\widetilde Y_s^\e)^\alpha} {U^\e_s - \widetilde U^\e_s}  \bI(R^\e_s>0)\,\di s\\
&= C_\e \int_0^t \frac{h_\e(U_s^\e) - h_\e(\widetilde U_s^\e)} {U^\e_s - \widetilde U^\e_s} \bI(R^\e_s>0)\,\di s,
\ea
where we have abbreviated $h_\e(y) = \big(F_\e^{-1}(y)\big)^\alpha$. 
Note that $h_\e(\cdot)$ is absolutely continuous with integrable derivative.
Using the relation
\ba
\frac{h_\e(U_s^\e) - h_\e(\widetilde U_s^\e)} {U^\e_s - \widetilde U^\e_s} = \int_0^1 h'_\e\big(\theta U_s^\e + (1-\theta) \widetilde U_s^\e\big)\,\di \theta,
\ea
we can write 
\ba
\int_0^t \frac{h_\e(U_s^\e) - h_\e(\widetilde U_s^\e)} {U^\e_s - \widetilde U^\e_s} \bI(R^\e_s>0)\,\di s 
\leq C_\e  \int_0^t \int_0^1  h'_\e\big(\theta U_s^\e + (1-\theta) \widetilde U_s^\e\big)\, \di \theta\,\di s .
\ea
For each $\theta \in [0,1]$, the process $R^{\e, \theta} := \theta U^\e + (1-\theta) \widetilde U^\e$ can be written as
\ba
R^{\e,\theta}_t=
\int_0^t \mu_s^{\e,\theta}\, \di B_s + \int_0^t \nu_s^{\e,\theta}\, \di W_s,
\ea
with the processes $\mu^{\e,\theta}$ and $\nu^{\e,\theta}$ given by
\ba
\mu^{\e,\theta}_t=\theta\Big(G_\e(Y^\e_t) + G_\e(\widetilde Y^\e_t)\Big)\quad\text{and}\quad
\nu^{\e,\theta}_t&=(1-\theta)\Big(H_\e(Y^\e_t) + H_\e(\widetilde Y^\e_t)\Big).  
\ea
One easily checks that $\frac12\leq |\mu_s^{\e,\theta}|^2 + |\nu_s^{\e,\theta}|^2 \le 1$,  
Therefore, $R^{\e,\theta}$ possesses a local time and for all $\theta \in[0,1]$, $x\in \bR$,
\ba
\E L^x_t(R^{\e,\theta})\le \sqrt{2}\E L^x_t(W) = 2\sqrt{\frac{t}{\pi}} \ex^{-x^2/2t}. 
\ea
Using Fubini's theorem and the occupation time formula we obtain the estimate
\ba
\E \int_0^1 \int_0^t h'_\e\big(\theta U_s^\e + (1-\theta) \widetilde U_s^\e\big) \,\di s\,\di \theta 
& = \E \int_0^1  \int_{\bR} h'_\e(x) L_t^x (R^{\e,\theta})\, \di x\,\di \theta\\
&\le C\sqrt{t} \int_0^1   \int_{\bR} h'_\e(x)\ex^{-x^2/2t}\, \di x \,\di \theta\le C_{t, \e}<\infty,
\ea
whence we get \eqref{e:int 1/z<oo}. 

In the following Lemma we prove a technical estimate that has been used in \eqref{e:gg}.
\begin{lem}
\label{l:G}
Let $\alpha\in (0,1)$, and let $G_\e$ and $H_\e$ be functions defined in \eqref{e:ggg}.
Then for any $\e>0$ there is $C_\e>0$ such that for all $y_1, y_2\in\bR$ we have 
\ba
\label{e:GH}
\bigl|G_\e(y_1) - G_\e(y_2)\bigr|^2+\bigl|H_\e(y_1) - H_\e(y_2)\bigr|^2 \le C_\e\Big| |y_1|^\alpha -|y_2|^{\alpha}\Big|^2.
\ea
\end{lem}
\begin{proof}
Denote for some $\beta_1,\beta_2\in\bR$
\ba
G_\e(y_1)=\sin \beta_1,\quad H_\e(y_1)=\cos \beta_1,\\
G_\e(y_2)=\sin \beta_2,\quad H_\e(y_2)=\cos \beta_2.
\ea
Then we have
\ba
|\sin \beta_1  - \sin \beta_2  |^2 +|\cos \beta_1  - \cos \beta_2 |^2
&=2 -2 \sin\beta_1\sin\beta_2 -2 \cos\beta_1\cos\beta_2\\
&=2(1-\cos(\beta_1-\beta_2))\\
&\leq  |\beta_1-\beta_2|^2\\
&=\Big|\arcsin G_\e(y_1) - \arcsin G_\e(y_2) \Big|^2\\
&=\Big|\arcsin \Big( G_\e(y_1)H_\e(y_2) - G_\e(y_2)H_\e(y_1)\Big) \Big|^2\\
&=\Big|\arcsin \frac{\e(|y_1|^\alpha - |y_2|^\alpha)}{\sigma_\e(y_1)\sigma_\e(y_2)} \Big|^2\\
&\leq \frac{\e^2||y_1|^\alpha - |y_2|^\alpha|^2}{\sigma_\e(y_1)^2\sigma_\e(y_2)^2}\wedge \frac{\pi^2}{4}\\
&\leq \frac{||y_1|^\alpha - |y_2|^\alpha|^2}{\e^2}\wedge 4\\
&\leq \frac{4\e^2+1}{\e^2} \cdot \Big(\Big||y_1|^\alpha - |y_2|^\alpha\Big|^2 \wedge 1 \Big),
\ea
and the inequality \eqref{e:GH} follows. 
\end{proof}

\section{Analysis of the SDE \eqref{e:mainn} with the bracket}

In this and further sections we restrict ourselves to the case $\alpha\in(0,1)$.

\subsection{Proof of Theorem \ref{t:itobracket}}

Our main goal in this section is to show that for $\alpha\in (0,1)$, the solution $Y^\e$ to \eqref{e:main+eps-ito} also solves 
\eqref{e:main+eps}, i.e.\ that $[|Y^\e|^\alpha, B]_t = \alpha \int_0^t (Y^\e_s)^{2\alpha-1}\,\di s$. 

\begin{lem}
\label{l:skobka}
Let $\e>0$ and let $Y^\e$ be a strong solution to \eqref{e:main+eps-ito}. Then, for any $\alpha\in (0,1)$ the quadratic variation
\ba
{} [|Y^\e|^\alpha, B]_t = \lim_{n\to\infty} \sum_{t_k\in D_n,t_k< t} 
\big(| Y^\e_{t_k}|^\alpha-|Y^\e_{t_{k-1}}|^\alpha\big)( B_{t_k}- B_{t_{k-1}})
\ea 
exists as a limit in u.c.p.\ and
\ba
{} [| Y^\e|^\alpha, B]_t  =\alpha \int_0^t ( Y^\e_s)^{2\alpha-1}\,\di s.
\ea
\end{lem}
\begin{proof}
Define the processes
\begin{align}
\label{e:W_12}
W^{1,\e}_t &= \frac{1}{\sqrt{2}} \int_0^t \frac{(|Y_s^\e|^\alpha + \e)\,\di B_s - (|Y_s^\e|^\alpha - \e)\,\di W_s}{\sigma_\e(Y_s^\e)},\\
\label{e:W1}
W^{2,\e}_t & = \frac{1}{\sqrt{2}} \int_0^t \frac{ (|Y_s^\e|^\alpha - \e )\,\di B_s + (|Y_s^\e|^\alpha + \e)\,\di W_s}{\sigma_\e(Y_s^\e)}.
\end{align}
It is easy to check that $W^{1,\e}$, $W^{2,\e}$ are independent standard Brownian motions and 
the original Brownian motions $B$ and $W$ 
satisfy the no-tilde counterparts of \eqref{e:wt B} and \eqref{e:wt W}, namely, we have
\ba
B_t &= \frac{1}{\sqrt{2}} 
\int_0^t \frac{\bigl(|Y^\e_s|^\alpha + \e\bigr)\,\di W^{1,\e}_s 
+ \bigl(|Y^\e_s|^\alpha - \e\bigr)\, \di W^{2,\e}_s}{\sigma_\e(Y_s^\e)},\\
W_t & = \frac{1}{\sqrt{2}} \int_0^t \frac{\bigr(\e-|Y^\e_s|^\alpha \bigr)\,\di W^{1,\e}_s 
+ \bigl(|Y^\e_s|^\alpha + \e\bigr)\,\di W^{2,\e}_s}{\sigma_\e(Y_s^\e)}.
\ea
The process
\ba
\widehat{W}^\e_t = \frac{1}{\sqrt{2}}(W^{1,\e}_t + W^{2,\e}_t)
\ea
is also a Brownian motion and we have the equality
\ba
\di \widehat W_t^\e = |Y_t^\e|^\alpha \,\di B_t + \e \di W_t. 
\ea
Literally repeating the argument of Section \ref{s:wexist} we get that the process 
$\widehat Y_t^\e=F_\e^{-1} \bigl( F_\e(x) + \widehat{W}_t^\e\bigr)$ 
solves \eqref{e:main+eps-ito}. 
Therefore, by the uniqueness Theorem~\ref{t:exuniqito}, $\widehat Y_t^\e=Y_t^\e= F_\e^{-1} ( F_\e(x) + \widehat{W}^\e_t)$ a.s. 

Let us apply the generalized It\^o formula from 
\cite{FPSh-95} to the function 
\ba
h_\e(z):=|F_\e^{-1} \bigl( F_\e(x) + z\bigr)|^\alpha, \quad h'_\e\in L^2_\mathrm{loc}(\bR) ,
\ea
and the Brownian motion $\widehat W^\e$:
\ba
|Y_t^\e|^\alpha=h_\e(\widehat W_t^\e)&=|x|^\alpha+ \int_0^t h'_\e(\widehat W_s^\e)\,\di \widehat W_s^\e 
+ \frac12 [h'_\e(\widehat W^\e), \widehat W^\e]_t\\
&=|x|^\alpha+ \alpha\int_0^t (  Y_s^\e )^{\alpha-1} 
\sigma_\e( Y_s^\e) \,\di \widehat W_s^\e + A^\e_t,
\ea
where
$A^\e_t=\frac12 [h'_\e(\widehat W^\e), \widehat W^\e]_t$
is an adapted zero energy process \cite[Theorem 3.5]{FPSh-95}. Now we decompose the bracket's partial sum as
\ba
\label{e:skobka |X|^alpha}
\sum_{t_k\in D_n,t_k< t} 
\Big(| Y_{t_k}^\e|^\alpha & -| Y_{t_{k-1}}^\e|^\alpha\Big)( B_{t_k}- B_{t_{k-1}})\\
& = \sum_{t_k\in D_n,t_k< t} \alpha\int_{t_{k-1}}^{t_k} (  Y_s^\e )^{\alpha-1} 
\sigma_\e( Y_s^\e)     \,\di \widehat W_s^\e \cdot ( B_{t_k}- B_{t_{k-1}})\\
&+
\sum_{t_k\in D_n,t_k< t} 
\Big(A_{t_k}^\e-A_{t_{k-1}}^\e\Big)( B_{t_k}- B_{t_{k-1}})=:J_1^\e+J_2^\e.
\ea
Since $A^\e$ is of zero energy, the second sum $J_2^\e$ converges to zero in probability by the Cauchy--Schwarz inequality.

We rewrite the first sum in terms of $W^{1,\e}$ and $W^{2,\e}$:
	\ba
	J_1^\e&=\sum_{t_k\in D_n,t_k< t}\frac\alpha2 \int_{t_{k-1}}^{t_k} (  Y_s^\e )^{\alpha-1} 
	\sigma_\e( Y_s)      \,\di W_s^{1,\e} \cdot 
	\int_{t_{k-1}}^{t_k}  \frac{|Y_s^\e|^\alpha + \e}{\sigma_\e( Y_s^\e)  }\,\di W^{1,\e}_s\\
	&+\sum_{t_k\in D_n,t_k< t}\frac\alpha2 \int_{t_{k-1}}^{t_k} (  Y_s^\e )^{\alpha-1} 
	\sigma_\e( Y_s)      \,\di W_s^{1,\e} \cdot 
	\int_{t_{k-1}}^{t_k}  \frac{|Y_s^\e|^\alpha - \e}{\sigma_\e( Y_s^\e)  }\,\di W^{2,\e}_s\\
	&+\sum_{t_k\in D_n,t_k< t}\frac\alpha2 \int_{t_{k-1}}^{t_k} (Y_s^\e )^{\alpha-1} 
	\sigma_\e( Y_s^\e)      \,\di W_s^{2,\e} \cdot 
	\int_{t_{k-1}}^{t_k}  \frac{|Y_s^\e|^\alpha + \e}{\sigma_\e( Y_s^\e)  }\,\di W^{1,\e}_s\\
	&+\sum_{t_k\in D_n,t_k< t}\frac\alpha2 \int_{t_{k-1}}^{t_k} (Y_s^\e )^{\alpha-1} 
	\sigma_\e( Y_s^\e)       \,\di W_s^{2,\e} \cdot 
	\int_{t_{k-1}}^{t_k}  \frac{|Y_s^\e|^\alpha - \e}{\sigma_\e( Y_s^\e)  }\,\di W^{2,\e}_s\\
	&:=\frac{\alpha}{2}\Big(I_{11}^\e+I_{12}^\e+I_{21}^\e+I_{22}^\e\Big).
	\ea
By Theorems II.23 and II.29 from \cite{Protter-04} we obtain that 
as $n\to\infty$
\ba
\label{e:convintprob}
I_{11}^\e&\stackrel{\P}{\to}
\int_0^t ( Y_s^\e )^{\alpha-1} \big(|Y_s^\e|^{\alpha} + \e\big)\,\di s,\\
I_{22}^\e&\stackrel{\P}{\to}
\int_0^t ( Y_s^\e )^{\alpha-1} \big(|Y_s^\e|^{\alpha} - \e\big)\,\di s,\\
I_{12}^\e&\stackrel{\P}{\to}0,\quad   I_{21}^\e\stackrel{\P}{\to}0,
\ea
and the statement of the Lemma follows.
\end{proof}

\subsection{Proof of Theorem \ref{t:final}}

To establish strong uniqueness of equation \eqref{e:main+eps}, we can show that any of its strong 
solutions also solves \eqref{e:main+eps-ito} and appeal to Theorem~\ref{t:exuniqito}. 
This is done under the additional requirement that the solution is a semimartingale.

\begin{rem}
	From Lemma~\ref{l:skobka} it follows that for any strong solution $Y^\e$ to \eqref{e:main+eps-ito}, the process $[|Y^\e|^\alpha,B]$ is absolutely continuous and hence of locally bounded variation.
\end{rem}
Recall that $\sigma_\e$ and $F_\e$ satisfy \eqref{e:sigma} and \eqref{e:Fe}. 
\begin{lem}\label{l:X^2}
$F_\e(X^\e)^2$ is a semimartingale with decomposition
\ba
F_\e(X_t^\e)^2 = F_\e(x)^2 + 2\int_0^t F_{\e}(X_s^\e) \frac{ |X_s^\e|^\alpha\, \di B_s + \e\, \di W_s }{\sqrt{|X_s^\e|^{2\alpha}+\e^2}} 
+  t + L_t^\e,\quad t\ge 0,
\ea
where $t\mapsto L_t^\e$ is a non-decreasing continuous process.
\end{lem}
\begin{proof}
Let $T>0$ and 
let $D_n = \{t_k^n = kT 2^{-n}, k = 0,\dots, 2^n\}$ be the dyadic partition of $[0,T]$.
The semimartingale $X^\e$ has quadratic variation equal to 
\ba
{}[X^\e]_t = \int_0^t |X_s^\e|^{2\alpha}\,\di s + \e\, t. 
\ea
Note also that $F_\e(\cdot)^2\in C^2(\bR,\bR)$ with 
\begin{align}
h_\e(z)&:=\frac{\di}{\di z}F_\e(z)^2 = \frac{2 F_\e(z)}{\sqrt{|z|^{2\alpha}+\e^2}},\\
\label{e:h2}
h'_\e(z)&:=\frac{\di^2}{\di z^2} F_\e(z)^2 = \frac{2 }{|z|^{2\alpha}+\e^2} - \frac{2\alpha \, F_\e(z) (z)^{2\alpha-1}}{\sqrt{(|z|^{2\alpha}+\e^2)^3}}.
\end{align}
Note that the last term in \eqref{e:h2} is well defined and continuous since $F(z)\sim z/\e$, $z\to 0$, see \eqref{e:Fasympt}.
By virtue of the generalized It\^o formula from \cite{FPSh-95} we get
\ba
\label{e:ito-eqn-F-eps}
F_\e(X_t^\e)^2 &=  F_\e(x)^2 + \int_0^t h_{\e}(X_s^\e)\,\di X_s^\e +  t 
-\alpha \int_0^t \frac{F_\e(X_s^\e) (X_s^\e)^{2\alpha-1}}{\sqrt{|X_s^\e|^{2\alpha}+\e^2}}\, \di s\\
& =  F_\e(x)^2 + \int_0^t h_{\e}(X_s^\e)\, \di X_s^\e +  t -\frac\alpha2 \int_0^t h_\e(X_s^\e) (X_s^\e)^{2\alpha-1}\, \di s,
\ea
where the first integral exists as a limit in u.c.p.:
\ba
\label{eq:intheps}
\int_0^t h_\e(X_s^\e)\, \di X_s^\e = \lim_{n\to \infty} \sum_{t^n_k\in D_n, t^n_k< t} h_\e(X_{t^n_{k-1}}^\e)\big(X_{t^n_{k}}^\e - X_{t^n_{k-1}}^\e\big). 
\ea
We expand
\ba
\sum_{t^n_k\in D_n, t^n_k< t} h_\e(X_{t^n_{k-1}}^\e)\big(X_{t^n_{k}}^\e - X_{t^n_{k-1}}^\e\big)
&{} = \sum_{t^n_k\in D_n, t^n_k< t} h_\e(X_{t^n_{k-1}}^\e)\int_{t^n_{k-1}}^{t^n_k} \big(|X_s^\e|^\alpha\, \di B_s + \e\, \di W_s\big)\\
&{} + \frac12 \sum_{t^n_k\in D_n, t^n_k< t} h_\e(X_{t^n_{k-1}}^\e) \big(A_{t^n_{k}}^\e - A_{t^n_{k-1}}^\e\big),
\ea
where $A^\e = [|X^\e|^\alpha, B]$. 
Since $h_\e (X^\e)$ is continuous, we have 
\ba
& \sum_{t^n_k\in D_n, t^n_k< t} h_\e(X_{t^n_{k-1}}^\e)\int_{t^n_{k-1}}^{t^n_k} \big(|X_s^\e|^\alpha\, \di B_s + \e\, \di W_s\big) 
\to \int_0^t h_\e(X_{s}^\e) \big(|X_s^\e|^\alpha\, \di B_s + \e\,  \di W_s\big),\quad  n\to\infty,
\ea
in u.c.p. Further, define 
\ba
\tau^{n,\e}_k = \min\{s\ge t^n_{k-1}\colon X_s^\e = 0\}\wedge t^n_k
\ea
and write 
\ba
&\sum_{t^n_k\in D_n, t^n_k< t} h_\e(X_{t^n_{k-1}}^\e) \big(A_{t^n_{k}}^\e - A_{t^n_{k-1}}^\e\big)\\
& = \sum_{t^n_k\in D_n, t^n_k< t} h_\e(X_{t^n_{k-1}}^\e) \big(A_{t^n_{k}}^\e - A_{\tau^{n,\e}_k}^\e\big) 
+ \sum_{t^n_k\in D_n, t^n_k< t} h_\e(X_{t^n_{k-1}}^\e) \big(A_{\tau^{n,\e}_{k}}^\e -A_{t^n_{k-1}}^\e\big). 
\ea
Note that $h_\e(X_{\tau^{n,\e}_k}^\e) = 0$ if $\tau^{n,\e}_k<t^n_k$ and $A_{t^n_{k}}^\e - A_{\tau^{n,\e}_k}^\e = 0$ 
if $\tau^{n,\e}_k=t^n_k$. Therefore,
\ba
&\Big|
\sum_{t^n_k\in D_n, t^n_k< t} 
h_\e(X_{t^n_{k-1}}^\e) \big(A_{t^n_{k}} - A_{\tau^{n,\e}_k}\big)\Big|^2
=\Big|\sum_{t^n_k\in D_n, t^n_k< t} 
\big(h_\e(X_{t^n_{k-1}}^\e) - h_\e(X_{\tau^n_k}^\e)\big) \big(A_{t^{n,\e}_{k}}^\e - A_{\tau^{n,\e}_k}^\e\big)\Big|^2\\
& \le \sum_{t^n_k\in D_n, t^n_k< t} 
\big(h_\e(X_{t^n_{k-1}}^\e) - h_\e(X_{\tau^{n,\e}_k}^\e)\big)^2 
\sum_{t^n_k\in D_n, t^n_k< t}\big(A_{t^n_{k}}^\e -A_{\tau^{n,\e}_k}^\e\big)^2\to 0, \quad n\to\infty,
\ea
thanks to the fact that $A^\e$ is of bounded variation on $[0,t]$.
On $[t^n_{k-1},\tau^{n,\e}_k]$, $X^\e$ satisfies a Stratonovich SDE with smooth coefficients, therefore (see Section V.5 in \cite{Protter-04})
\ba
\label{e:ZZ}
A_{u}^\e - A_{t^n_{k-1}}^\e = \alpha \int_{t^n_{k-1}}^{u}(X_s^\e)^{2\alpha -1}\,\di s,\quad  u\in[t^n_{k-1},\tau^{n,\e}_{k}].
\ea
Fix $m\ge 1$ and consider $n> m$. Let 
\ba
D_m^\e = \{t_j^m \in D_m\colon  \tau^{m,\e}_{j}= t^{m}_{j}\}
\ea
be the points of partition $D_m$ such that $X^\e$ does not vanish on the intervals $[t^m_{j-1},t^m_j]$. Let
\ba
P_{m,n}^\e = 
\bigcup_{t_{j}^m \in D_m^\e} P_{m,n,j} = 
\bigcup_{t_{k}^m \in D_m^\e}  \{t^n_k\in D_n\colon  t^n_k\in [t^{m}_{j-1},t^{m}_{j}]\}
\ea
be the corresponding points of the finer partition $D_n\supseteq D_m$. Clearly,
\ba
\label{e:Ktm}
\sum_{t^n_k\in P_{m,n}^\e, t^n_k< t} h_\e(X_{t^n_{k-1}}^\e) &\big(A_{\tau^{n,\e}_{k}}^\e - A_{t^n_{k-1}}^\e\big) 
= \sum_{t^m_j\in D_{m}^\e, t_j^m<t}\sum_{t_n^k\in P_{m,n,j}^\e} h_\e(X_{t^n_{k-1}}^\e)\cdot \alpha\int_{t^n_{k-1}}^{t_k^n}(X_s^\e)^{2\alpha -1}\,\di s \\  
&\to \sum_{t^n_j\in D_{m}^\e, t_j^m<t} \alpha \int_{t^m_{j-1}}^{t^m_j }h_{\e}(X_s^\e)(X_s^\e)^{2\alpha -1}\,\di s=:K^{m,\e}_t, \quad n\to\infty . 
\ea
Further, using again \eqref{e:ZZ} we get
\ba
\sum_{t^n_k\in D_n\setminus P_{m,n}^\e, t^n_k< t} h_\e(X_{t^n_{k-1}}^\e) \big(A_{\tau^{n,\e}_{k}} - A_{t^n_{k-1}}^\e\big)=
\sum_{t^n_k\in D_n\setminus P_{m,n}^\e, t^n_k< t} h_\e(X_{t^n_{k-1}}^\e)\cdot  \alpha  \int_{t^n_{k-1}}^{\tau^{n,\e}_{k}}(X_s^\e)^{2\alpha -1}\,\di s 
\ea
is a non-decreasing process. 
As $n\to\infty$, it converges to some non-negative non-decreasing process, say, $L^{m,\e}$, and we get
\ba
\label{e:F2lim}
F_\e(X_t^\e)^2 = F_\e(x)^2  &+ 2\int_0^t h_{\e}(X_s^\e)\big( |X_s^\e|^\alpha\, \di B_s + \e\, \di W_s \big)\\ 
&+  t +  K_t^{m,\e} + L^{m,\e}_t -\frac\alpha2 \int_0^t h_\e(X_s^\e) (X_s^\e)^{2\alpha-1}\,\di s.
\ea
Moreover, since the set of zeros of $X^\e$ is closed and $h_\e(z)(z)^{2\alpha-1} = 0$ for $z=0$, we get 
\ba
K^{m,\e}_t \to \frac\alpha2 \int_0^t \bI(X_s^\e\neq 0) h_{\e}(X_s^\e)(X_s^\e)^{2\alpha -1}\,\di s 
= \frac\alpha2 \int_0^t h_{\e}(X_s^\e)(X_s^\e)^{2\alpha -1}\, \di s,\quad m\to\infty.
\ea
As a result, the processes $L^{m,\e}$ also monotonically converge as $m\to\infty$ 
to some non-decreasing limit $L^\e$. Hence, passing to the limit in \eqref{e:F2lim}, 
we get the desired statement. 	
\end{proof}

\begin{lem}
\label{l:negativemoments}
For any $a\in (-1,0)$ and $t\in[0,T]$, there is $C_{a,\e,T}>0$ such that
\ba
\E |X_t^\e|^a \le C_{a,\e,T}. 
\ea
Moreover, $X^\e$ spends zero time at $0$.
\end{lem}
\begin{proof}
By Lemma~\ref{l:X^2}, $Q_t^\e:=F_\e (X_t^\e)^2$ solves the equation
\ba
Q_t^\e = F_\e(x)^2 + 2 \int_0^t \sqrt{Q_s^\e}\, \di \widehat W_s^\e + t + L_t^\e,
\ea
where 
\ba
\widehat W_t^\e = \int_0^t \frac{\sign \big(F_{\e}(X_s^\e)\big)\big( |X_s^\e|^\alpha\, \di B_s + \e\, \di W_s \big)}{\sqrt{|X_s^\e|^{2\alpha}+\e^2}}
\ea 
is a standard Wiener process. 
Since $L^\e$ is non-decreasing and continuous, a slight modification of the comparison Theorem 3.7 in 
Chapter IX in \cite{RevuzYor05} yields 
that $Q_t^\e\ge \widehat{Q}^\e_t$, $t\geq 0$, with probability 1, where $\widehat{Q}^\e$ is the unique strong solution to the equation
\ba
\widehat Q_t^\e = F_\e(x)^2 + 2 \int_0^t \sqrt{\widehat Q_s^\e}\, \di \widehat W_s^\e + t,
\ea
and hence is a square of a standard Wiener process started at $F_\e(x)$. 
Consequently, $Q^\e$ spends zero time at zero.
Furthermore, noting that for $|F_\e (x)|\leq C_{\e} |x|$, $x\in\bR$, for some $C_\e>0$ (recall \eqref{e:Fasympt}), we obtain
for $a>-1$
\ba
\E |X_t^\e|^a \le C_{\e}^a \E |F_\e(X_t^\e)|^a = C_{\e}^a\E |Q_t^\e|^{a/2} \leq  C_{\e}^a\E |\widehat Q_t^\e|^{a/2} 
=C_{\e}^a\E |F_\e(x)+\widehat W_t^\e|^{a}
\le  C_{\e,a,T}.
\ea
\end{proof}

\begin{lem}
Let $X^\e$ be an It\^o semimartingale solution of the equation \eqref{e:main+eps}. Then 
\ba
{}[|X^\e|^\alpha,B]_t=\alpha\int_0^t (X^\e_s)^{2\alpha-1}\,\di s.
\ea
\end{lem}
\begin{proof}
Without loss of generality assume that $x=0$. Denote $A_t^\e=[|X^\e|^\alpha,B]_t$. Since 
$X^\e$ is an It\^o semimartingale, there is a progressively measurable process $a^\e=(a^\e_t)_{t\geq 0}$ such that
\ba
A^\e_t=\int_0^t a^\e_s\,\di s.
\ea
 For each $n\geq 1$ consider the stopping times 
\ba
\sigma^n_0&=0,\\
\tau^{n,\e}_k&=\inf\Big\{t > \theta^{n,\e}_{k-1}\colon |X^\e_t|=\frac1n \Big\},\\
\sigma^{n,\e}_k&=\inf\{t >\tau^{n,\e}_k\colon X^\e_t=0 \},\quad k\geq 1.
\ea
Denote
\ba
\theta^{n,\e}_t:=\sum_{k} \bI_{(\sigma_{k-1}^{n,\e},\tau_k^{n,\e} ]}(t),
\ea
and note that $\theta^{n,\e}_t\in[0,1]$ for all $n,\e$ and $t$.
By Lemma \ref{l:negativemoments}, since $X^\e$ spends zero time at $0$, we get
\ba
\E \int_0^t \theta^{n,\e}_s\,\di s
\leq \E \int_0^t \bI(|X_s^\e|\leq 1/n)\,\di s \to 0,\quad n\to\infty. 
\ea
Hence
\ba
A_t^\e &= \sum_k \alpha \int_{\tau^{n,\e}_k\wedge t}^{\sigma^{n,\e}_k\wedge t} (X_s^\e)^{2\alpha-1}\,\di s 
+ \sum_k \Big(A_{\tau^{n,\e}_k\wedge t}^\e - A_{\sigma^{n,\e}_{k-1}\wedge t}^\e\Big)\\
&=\alpha \int_0^t (X_s^{\e})^{2\alpha-1}\,\di s 
-\sum_k \alpha \int_{\sigma^{n,\e}_k}^{\tau^{n,\e}_k} (X_s^{\e})^{2\alpha-1}\,\di s 
+ \int_0^t \theta^{n,\e}_s\, \di A_s^\e\\
&=\alpha \int_0^t (X_s^\e)^{2\alpha-1}\,\di s - S_1^{n,\e}(t)+S_2^{n,\e}(t).
\ea
Clearly, by Lemma \ref{l:negativemoments} and the dominated convergence theorem 
\ba
\E |S_1^{n,\e}(t)|\leq \alpha \E \int_0^t |X_s^\e|^{2\alpha-1}\bI(|X_s^\e|\leq 1/n)\,\di s\to 0,\quad n\to\infty, 
\ea
and 
\ba
|S_2^{n,\e}(t)|\leq \int_0^t \theta^{n,\e}_s |a^\e_s|\, \di s\to 0,\quad n\to\infty\quad \text{a.s.}
\ea
Since $|\theta^{n,\e}_s a^\e_s| \leq  |a^\e_s|$, the dominated convergence theorem finishes the proof.
\end{proof}

\section{Proof of Theorem \ref{t:selection}}

Recall that 
\ba
X^0_t:=\Big((1-\alpha)B_t- (x)^{1-\alpha}\Big)^{\frac{1}{1-\alpha}},\quad t\geq 0, 
\ea
is the strong solution of the equation \eqref{e:mainn}
that spends zero time at zero and has no skew behaviour at zero.   
Let $X^\e$ be the semimartingale solution of \eqref{e:main+eps} or \eqref{e:main+eps-ito}.
We apply the generalized It\^o formula from \cite{krylov2008controlled} to get
\ba
F_\e(X^\e_t)&=F_\e(x) + \int_0^t \frac{|X^\e_s|^\alpha\,\di B_s + \e\,\di W_s}{ \sqrt{|X^\e_s|^{2\alpha}+\e^2}}\\
&+\frac{\alpha}{2}\int_0^t \frac{(X^\e_s)^{2\alpha-1}\,\di s}{ \sqrt{|X^\e_s|^{2\alpha}+\e^2}}
-\frac{\alpha}{2}\int_0^t \frac{|X^\e_s|^{2\alpha}(X^\e_s)^{2\alpha-1}\,\di s}{ (|X^\e_s|^{2\alpha}+\e^2)^{3/2}}
-\frac{\alpha}{2}\int_0^t \frac{\e^2 (X^\e_s)^{2\alpha-1}\,\di s}{ (|X^\e_s|^{2\alpha}+\e^2)^{3/2}}\\
&=B_t- \int_0^t \frac{\e \,\di B_s }{ \sqrt{|X^\e_s|^{2\alpha}+\e^2}} + \int_0^t \frac{\e \,\di W_s }{ \sqrt{|X^\e_s|^{2\alpha}+\e^2}}\\
&=B_t+I_t^{\e}.
\ea
Equivalently, we have
\ba
X^\e_t=F_\e^{-1}(B_t+I_t^{\e}).
\ea
Since 
\ba
F_\e^{-1}(x)\to F_0^{-1}(x)= (1-\alpha)^{\frac{1}{1-\alpha}}(x)^{\frac{1}{1-\alpha}}\quad \text{as}\quad\e\to 0
\ea
uniformly on each compact interval, to prove the Theorem it is sufficient to show that
$I^{\e}$ converges to zero in u.c.p.

By the Doob inequality, for any $t>0$
\ba
\E \sup_{s\in [0,t]}|I_s^{\e} |^2 &\leq 
2 \E \sup_{s\in [0,t]}\Big|\int_0^s \frac{\e \,\di B_u }{ \sqrt{|X^\e_u|^{2\alpha}+\e^2}}\Big|^2
+2 \E \sup_{s\in [0,t]}\Big|\int_0^s \frac{\e \,\di W_u }{ \sqrt{|X^\e_u|^{2\alpha}+\e^2}}\Big|^2\\
&\leq 
16  \E \int_0^t\frac{\e^2 \,\di s }{|X^\e_s|^{2\alpha}+\e^2}.
\ea
To evaluate the latter expectation we recall that due to \eqref{e:wt X} 
$X^\e\stackrel{\di }{=} F^{-1}_\e(F_\e(x) + \widehat W )$ for some Brownian motion $\widehat W$. Let us show that
\ba
\label{e:to0} 
\E \int_0^t\frac{\e^2 \,\di s }{|X^\e_s|^{2\alpha}+\e^2}
&=  \E \int_0^t\frac{\e^2}{|F^{-1}_\e(F_\e(x) + \widehat W_s ))|^{2\alpha}+\e^2} \,\di s\\
&=\int_0^t \frac{1}{\sqrt{2\pi s}}\int_{-\infty}^\infty \frac{\e^2}{|F^{-1}_\e(y)|^{2\alpha}+\e^2}\ex^{-\frac{(y-F_\e(x))^2}{2s}}\,\di y \,\di s
\to 0,\quad \e\to 0.
\ea
First we note that for all $y>0$ and $\e>0$
\ba
F_\e(y)\leq \int_0^y \frac{\di z}{z^\alpha}=\frac{y^{1-\alpha}}{1-\alpha}=F_0(y)
\ea
and hence
\ba
|F_\e(y)|\leq |F_0(y)| ,\quad y\in\bR,
\ea
and
\ba
|F_\e^{-1}(y)|\geq | F_0^{-1}(y)|= (1-\alpha)^{\frac{1}{1-\alpha}}|y|^{\frac{1}{1-\alpha}},\quad y\in\bR.
\ea
For $x\geq 0$ and $y\in\bR$ we have  
$0\leq F_\e(x)\leq F_0(x)$ get the following estimate:
\ba
\label{e:x1}
\frac{1}{\sqrt{2\pi t}}\frac{\e^2}{|F^{-1}_\e(y)|^{2\alpha}+\e^2}&\ex^{-\frac{(y-F_\e(x))^2}{2t}}\\
&\leq 
\frac{1}{\sqrt{2\pi t}}\Big( 
\ex^{-\frac{y^2}{2t}}\bI_{(-\infty,0)}(y)
+
\bI_{[0,F_0(x))}(y)+\ex^{-\frac{(y-F_0(x))^2}{2t}}\bI_{[F_0(x),\infty)}(y)\Big)
\ea
The right hand side of \eqref{e:x1} is integrable on $(t,y)\in(0,T]\times \bR$. 
Since for each $t\in(0,T]$ and $y\neq 0$
\ba
\frac{1}{\sqrt{2\pi t}}\frac{\e^2}{|F^{-1}_\e(y)|^{2\alpha}+\e^2}\ex^{-\frac{(y-F_\e(x))^2}{2t}}
\leq 
\frac{1}{\sqrt{2\pi t}}\frac{\e^2}{| (1-\alpha)y |^{2\alpha/(1-\alpha)}+\e^2} \to 0,\quad \e\to 0,
\ea
the limit \eqref{e:to0} follows by the dominated convergence theorem.
Since $F_\e$ is asymmetric, the limit \eqref{e:to0} holds for $x<0$, too.

\medskip
\noindent
 \textbf{Acknowledgments.}
 The authors thank 
the German Research Council (grant Nr.\  PA 2123/6-1) for financial support.  
The authors are grateful to the anonymous referees for their helpful reports.

%
%


\end{document}